\documentclass[a4paper,reqno,oneside]{amsart}

\usepackage{amsthm,amsfonts,amsmath,cite}

\usepackage{mathptmx}

\renewcommand{\Bar}{\overline}

\newcommand{\RR}{\mathbb{R}}
\newcommand{\CC}{\mathbb{C}}
\newcommand{\NN}{\mathbb{N}}
\newcommand{\ZZ}{\mathbb{Z}}

\newcommand{\cK}{\mathcal{K}}
\newcommand{\cG}{\mathcal{G}}
\newcommand{\cH}{\mathcal{H}}
\newcommand{\ELL}{\mathcal{L}}

\newtheorem{theorem}{Theorem}
\newtheorem{prop}[theorem]{Proposition}
\newtheorem{lemma}[theorem]{Lemma}

\theoremstyle{remark}

\newtheorem{remark}[theorem]{Remark}

\theoremstyle{definition}

\DeclareMathOperator{\spec}{spec}
\DeclareMathOperator{\dom}{dom}
\DeclareMathOperator{\Res}{Res}

\begin{document}

\title[]{A remark on the discriminant of Hill's equation and~Herglotz functions}

\author{Konstantin Pankrashkin}

\address{Laboratoire de math\'ematiques -- UMR 8628, Universit\'e Paris-Sud, B\^atiment 425, 91405 Orsay Cedex, France}

\email{konstantin.pankrashkin@math.u-psud.fr}

\begin{abstract} We establish a link between the basic properties of the discriminant
of periodic second-order differential equations and an elementary analysis of Herglotz functions.
Some generalizations are presented
using the language of self-adjoint extensions.

\bigskip

\noindent \it This is a preliminary version. The final version will be published in Archiv der Mathematik (Basel).

\end{abstract}



\maketitle

\section{Introduction}

Let $Q$ be a continuous, real-valued and $1$-periodic function on $\RR$.
The associated \emph{discriminant} (also called \emph{Lyapunov function}
or \emph{characteristic function}) $\Delta:\CC\to\CC$ is defined by
$\Delta(\lambda):=s'(1;\lambda)+c(1;\lambda)$, where $s$ and $c$ are the solutions to the Hill's differential equation
$-u''+Qu=\lambda u$ with  $s(0;\lambda)=c'(0;\lambda)=0$ and $s'(0;\lambda)=c(0;\lambda)=1$.
An important role in the study of Hill's equation is played by the set $\Sigma:=\Delta^{-1}\big((-2,2)\big)$;
in particular, its closure is exactly the spectrum of the Hill operator $L=-d^2/dx^2+Q$ in $L^2(\RR)$.
The discriminant $\Delta$ has remarkable oscillation properties, which are used to understand the structure of $\Sigma$.
More precisely, let us denote by $\mu_n$, $n\ge 1$, the increasing sequence of the Dirichlet
eigenvalues of the operator $u\mapsto -u''+Qu$ on the interval $(0,1)$; alternatively, $\mu_n$ can be viewed as the zeros of the function
$z\mapsto s(1;z)$; then:
\begin{itemize}
\item[(P1)] For any $\lambda\in\RR$ with $\big|\Delta(\lambda)\big|<2$ one has $\Delta'(\lambda)\ne 0$.
\item[(P2)] If $|\Delta(\lambda)|=2$ and $\Delta'(\lambda)=0$, then $\lambda$ coincides with
one of the Dirichlet eigenvalues,
\item[(P3)] For any $n\in\NN$ one has: \emph{either} $\Delta(\mu_n)\le -2$ and $\Delta(\mu_{n+1})\ge 2$
\emph{or}  $\Delta(\mu_n)\ge 2$ and $\Delta(\mu_{n+1})\le -2$.
\item[(P4)] If $\Delta(\mu_n)= 2$ and $\Delta'(\mu_n)=0$ for some $n$, then $\Delta''(\mu_n)< 0$.
\item[(P5)] If $\Delta(\mu_n)=-2$ and $\Delta'(\mu_n)=0$ for some $n$, then $\Delta''(\mu_n)> 0$.
\item[(P6)] The situation of (P4) holds iff all solutions of $-u''+Qu=\mu_n u$  are periodic,
  $u(x+1)=u(x)$.
\item[(P7)] The situation of (P5) holds iff all solutions of $-u''+Qu=\mu_n u$ are anti-periodic, $u(x+1)=-u(x)$.
\end{itemize}
Using (P1)--(P5) one easily can easily see that $\Sigma$ is the union of non-intersecting intervals
$(\alpha_n,\beta_n)$ such that $\mu_n\le \alpha_n<\beta_n\le \mu_{n+1}$. The intervals
$(\beta_n,\alpha_{n+1})$, if non-empty, are usually called \emph{gaps} or \emph{instability intervals}, and
one is often interested in situations in which some instability intervals are empty, and this is
characterized by the properties (P6) and (P7).

The study of the discriminant goes back to Lyapunov~\cite{lyap} who used the fact that
$\Delta$ belongs to a certain class of entire functions. M.~G.~Krein~\cite{Kr2,Kr1} made this approach more explicit,
but asked whether more elementary proofs can be found.
The approach usually presented in textbooks uses the machinery which is specific
to the ordinary differential equations such as oscillation theorems, zeros of the eigenfunctions and the variation of constants,
see e.g. Section VIII.3 in~\cite{CL}, Section 5.6 in~\cite{GTe} or Sections~1.6 and~2.4 and
the related historical remarks in~\cite{brown}, and the analysis of the derivatives of $\Delta$ is particularly involved.
Numerous works suggested various alternative approaches to the analysis of the discriminant, see e.g.~\cite{AC,haupt,HH,mn}.
In the present note we show that there is a simple link between the discriminant
and the holomorphic functions mapping the upper half of the complex plane to itself;
this observation does not seem to be well known. This relationship allows one to give
a new way of deriving the above properties. The new proof seems to be much more elementary
that those available in the literature, it is self-contained, and it does not involve any long
computation and does not use the theory of entire functions.
This is presented in the two subsequent sections, and we tried to keep this part of the presentation
as elementary as possible. Besides the pedagogical value, the approach is suitable
for the study of a larger class of periodic operators for which the previous methods
cannot be applied directly, and this is discussed in the last section.
We remark that after the completion of the present paper we learned about the preprint~\cite{mn}
in which a similar approach was developed.

\section{Some properties of Herglotz functions}

Recall that a  \emph{Herglotz function} is a holomorphic map $h:\CC\setminus\RR \to \CC$
satisfying for all $z$ the properties $h(\Bar z)=\overline{h(z)}$ and $\Im z \cdot \Im h(z)\ge 0$.
Such functions appear in various domains
and are closely related with the measure theory, the spectral theory and the harmonic analysis,
sometimes under different names like Nevanlinna, Stieltjes of $R$-functions,
see e.g. \cite[Section 29]{AG} or~\cite{GT,Arl}.
At reader's convenience we include complete proofs of the following two elementary assertions
which are entirely based on the definition.

\begin{lemma}\label{lem1}
Let $h$ be a non-constant Herglotz function, then:
\begin{itemize}
 \item[(A)] If $h$ is holomorphic near some $\lambda\in\RR$, then  $h'(\lambda)> 0$,
 \item[(B)] If $\lambda\in\RR$ is a pole of $h$, then it is a simple pole, and
 $\Res(h,\lambda)<0$.
\end{itemize}
\end{lemma}

\begin{proof} (A) Take the minimal $k\in\NN$ for which $\alpha:=h^{(k)}(\lambda)\ne 0$.
Clearly, $h(\lambda)$ and $\alpha$ are real. For $\theta\in[-\pi,\pi]$ and $r>0$
we will write $z:=\lambda+r e^{i\theta}$. The Taylor expansion
near $\lambda$ shows that for small $r$ 
the signs of $\Im h(z)$  and $\Im (\alpha r^k e^{ik\theta})=\alpha r^k \sin(k\theta)$ coincide
provided $k\theta\notin\pi\ZZ$.
Let $\theta:=\pi/(2k)$, then $\Im z>0$, which gives $\alpha>0$.
Assume now that $k>1$ and take $\theta:=3\pi/(2k)$, then
$\Im z>0$ and $\alpha r^k \sin(k\theta)=-\alpha r^k$, which gives $\Im h(z)<0$.
This contradiction shows that $k=1$ and proves (A).
In particular, we see that all zeros of $h$ are simple. To prove (B)
it is sufficient to apply (A) to the function $g(z):=-1/h(z)$ which is also Herglotz.
\end{proof}

\begin{lemma}\label{prop1}
Let $I=(\lambda_1,\lambda_2)\subset\RR$ be a non-empty open interval.
Let two functions $m,n$ be holomorphic in $(\CC\setminus\RR)\cup I$ and take real values in $I$.
Assume that
\begin{itemize}
\item $n$ is non-constant,
\item there are $a,b\in\RR$ with $a<b$ such that the two functions
$h_a(z):=\dfrac{m(z)-a}{n(z)}$ and $h_b(z):=\dfrac{m(z)-b}{n(z)}$
are Herglotz and non-constant,
\end{itemize}
then:
\begin{enumerate}
\item[(a)] The zeros of $n$ in $I$ are simple.
\item[(b)] If $\lambda\in I$ is such that $m(\lambda)\in(a,b)$, then $n(\lambda)m'(\lambda)>0$.
\item[(c)] If $n(\lambda)=0$ for some $\lambda\in I$, then $m(\lambda)\notin(a,b)$.
\item[(d)] Let $\mu$ and $\nu$ be successive zeros of $n$ in $I$, then
\emph{either}  $m(\mu)\le a$ and $m(\nu)\ge b$
\emph{or} $m(\mu)\ge b$ and $m(\nu)\le a$.
\item[(e)] If $\lambda\in I$ is such that $m'(\lambda)=0$, then:
\begin{itemize}
\item[(i)] if $m(\lambda)$ coincides with $a$ or $b$, then $n(\lambda)=0$.
\item[(ii)] if $m(\lambda)= a$, then $m''(\lambda)>0$.
\item[(iii)] if $m(\lambda)= b$, then $m''(\lambda)<0$.
\end{itemize}
\end{enumerate}
\end{lemma}

\begin{proof} (a) As the functions $m-a$ and $m-b$ cannot vanish simultaneously,
each zero of $n$ is a pole for $h_a$ or for $h_b$ (or for both). The simplicity of poles proved
in Lemma~\ref{lem1}(B) gives the result.

(b) By contradiction, assume that $n(\lambda)=0$, then $n'(\lambda)\ne 0$ by Lemma~\ref{lem1}(A), and
the residues
\begin{equation}
     \label{eq-res1}
\Res(h_a,\lambda)=\dfrac{m(\lambda)-a}{n'(\lambda)}, \quad
\Res(h_b,\lambda)=\dfrac{m(\lambda)-b}{n'(\lambda)}.
\end{equation}
have opposite signs, which contradicts Lemma~\ref{lem1}(B). Hence, $n(\lambda)\ne 0$ and
\begin{equation}
  \label{eq-hah}
       h'_a(\lambda)=\dfrac{m'(\lambda)}{n(\lambda)}-\dfrac{n'(\lambda)}{n(\lambda)^2}\big(m(\lambda)-a\big),
\quad
h'_b(\lambda)=\dfrac{m'(\lambda)}{n(\lambda)}-\dfrac{n'(\lambda)}{n(\lambda)^2}\big(m(\lambda)-b\big).
\end{equation}
If $m'(\lambda)=0$, then $h'_a(\lambda)h'_b(\lambda)\le 0$, which contradicts Lemma~\ref{lem1}(A).
As the second terms in the expressions for $h'_a$ and $h'_b$ are of opposite signs and both
$h'_a$ and $h'_b$ are positive by Lemma~\ref{lem1}(A), the first fraction must be positive.

(c) is a direct corollary of (b).

(d) Due to the simplicity of zeros proved in (a), the numbers $n'(\mu)$ and $n'(\nu)$
have opposite signs.

Assume first that none of the numbers $m(\mu)$ and $m(\nu)$ coincides
with $a$ or $b$. It follows from \eqref{eq-res1} and Lemma~\ref{lem1}(B)
that $\big(m(\mu)-a\big)/n'(\mu)<0$ and $\big(m(\nu)-b\big)/n'(\nu)<0$, and this means
that $m(\mu)-a$ and $m(\nu)-b$ have opposite signs.
By (c), the numbers $m(\mu)$ and $m(\nu)$ do not belongs to $(a,b)$, which means that
either $m(\mu)<a$ and $m(\nu)>b$ or  $m(\mu)>b$ and $m(\nu)<a$.

Assume now that $m(\mu)=a$. Using \eqref{eq-res1} and Lemma~\ref{lem1}(B)
we obtain $\Res(h_b,\mu)=(a-b)/n'(\mu)<0$, which shows that $n'(\mu)>0$ and $n'(\nu)<0$.
If $m(\nu)=b$, the assertion holds. Otherwise \eqref{eq-res1} and Lemma~\ref{lem1}(B) give
$\Res(h_b,\nu)=\big(m(\nu)-b\big)/n'(\nu)<0$, showing that $m(\nu)-b>0$.
The case $m(\mu)=b$ is considered similarly.

(d) To prove (i) we note that if $n(\lambda)$ was non-zero, then, by~\eqref{eq-hah},
 $h'_a(\lambda)$ or $h'_b(\lambda)$ would vanish, which contradicts Lemma~\ref{lem1}(A).
Let us prove (ii); the reasoning is similar for (iii).
By~Lemma~\ref{lem1}(B), $\Res(h_b,\nu_j)=(a-b)/n'(\lambda)<0$,
which gives $n'(\lambda)>0$. Finally, combining the Taylor expansion
$h_a(z)=\dfrac{m''(\lambda)}{2n'(\lambda)}(z-\lambda) + O(|z-\lambda|^2)$ for$z$ is near $\lambda$.
with Lemma~\ref{lem1}(A) we obtain $m''(\lambda)>0$.
\end{proof}

\begin{remark}
It should be noted that the assumptions of Lemma~\ref{prop1} themselves do not imply the existence of at least two zeros of $n$,
as illustrated by the functions $z\mapsto (z-a)/z$ which are Herglotz for any $a>0$.
\end{remark}

\section{Application to the discriminant}

We are now going to show how the properties (P1)--(P7) can be deduced with the help of Lemma~\ref{prop1}.

For any $\xi\in\CC^2$ and any $z\notin\{\mu_n: n\in\NN\}$
there is a  unique solution of $-u''+Qu=zu$ with $\big(u(0), u(1)\big)=\xi$,
\[
u(x)=u(0)\dfrac{s(1;z)c(x;z)-c(1;z)s(x;z)}{s(1;z)}+u(1)\dfrac{s(x;z)}{s(1;z)}.
\]
As the Wronskian $w(x;z):=s(x;z)c'(x;z)-s'(x;z)c(x;z)$ is constant in~$x$, we infer from
$w(1;z)=w(0;z)=-1$ that
\begin{equation}
       \label{eq-mum}
\begin{pmatrix} u'(0) \\ -u'(1) \end{pmatrix}=M(z)      
\begin{pmatrix} u(0) \\ u(1) \end{pmatrix}, \quad
M(z):=\dfrac{1}{s(1;z)}\begin{pmatrix}
-c(1;z) & 1\\
1 & -s'(1;z)
\end{pmatrix}.
\end{equation}
For $h(z):=\langle M(z)\xi, \xi\rangle\equiv u'(0)\overline{u(0)}-u'(1)\overline{u(1)}=-\displaystyle\int_0^1 \big(u' \Bar u)'(x)dx$ we have
\[
\Im h(z) =-\Im \int_0^1 \Big( \big|u'(x)\big|^2+\big(Q(x)- z\big)\big|u(x)\big|^2\Big)dx= \Im z \int_0^1 \big|u(x)\big|^2dx.
\]
Together with $h(\Bar z)=\overline{h(z)}$ this shows
that $h$ is Herglotz. Furthermore, $\Im h(z)=0$ for real $z$ different from all $\mu_n$, and
for any non-zero $\xi$ we have $\Im h(z)\ne 0$ for $z\notin\RR$, i.e. $h$ is non-constant.
Choosing $\xi=(1,\pm 1)$ we see that the functions $z\mapsto h_\pm(z):=-\big(\Delta(z)\pm 2\big)/s(1;z)$ are Herglotz and non-constant,
and the assertions (P1)--(P5) follow immediately from Lemma~\ref{prop1} with $I=\RR$.

Let us show (P6); the reasoning for (P7) is similar. Clearly, all solutions are periodic if
and only if $s$ and $c$ are periodic. From $w(1,\mu_n)=-1$ and $s(1;\mu_n)=0$ we infer
$s'(1;\mu_n)c(1;\mu_n)=1$. Together with $\Delta(\mu_n)=2$ this gives $s'(1;\mu_n)=c(1;\mu_n)=1$. In particular,
$s(0;\mu_n)=s(1;\mu_n)=0$ and $s'(0;\mu_n)=s'(1;\mu_n)=1$, i.e. $s$ is periodic.
As $c(0;\mu_n)=c(1;\mu_n)=1$ and $c'(0;\mu_n)=0$, the function $c$ is periodic iff $c'(1;\mu_n)=0$.
From $w(1;z)=-1$ we obtain
\[
c'(1;\mu_n)=\lim_{z\to\mu_n} c'(1;z)=\lim_{z\to\mu_n} \dfrac{s'(1;z)c(1;z)-1}{s(1;z)}
=
\lim_{z\to\mu_n} \dfrac{\big(s'(1;z)-1\big)\big(c(1;z)-1\big)}{s(1;z)}
+
\lim_{z\to\mu_n} \dfrac{\Delta(z)-2}{s(1;z)}.
\]
Recall that the zeros of $z\mapsto s(1;z)$ are simple due to Lemma~\ref{prop1}(1).
As both factors in the numerator of the first fraction vanish at $\mu_n$,
the first limit is zero. On the other hand, the second limit equals zero
iff $\mu_n$ is a degenerate zero of $\Delta-2$, which was exactly the claim.

\section{Discriminant and Weyl functions}

In the present section we would like to show how the discriminant $\Delta$
and the associated functions $h_\pm$ appear in an abstract framework
involving the spectral theory of self-adjoint extensions and the boundary triples.
A~very brief account of this machinery is given below, for a more detailed presentation
we refer to the papers~\cite{BGP08,DM91,GG}; an introduction can also be found in
the recent textbook~\cite[Chapters~14 and~15]{ksm}.

Let $\cH$ be a Hilbert space and $S$ be a densely defined symmetric operator in $\cH$.
It is known that $S$ has self-adjoint extensions iff one can construct
a \emph{boundary triple} composed from an auxiliary Hilbert space $\cG$ and two linear maps
$\Gamma,\Gamma'\dom S^*\to\cG$ with the following two properties: (i)
$\langle S^*f,g\rangle-\langle f,S^*g\rangle=\langle \Gamma'f,\Gamma g\rangle -
\langle \Gamma f,\Gamma' g\rangle$ for all $f,g\in\dom S^*$, and (ii)
the mapping $\dom S^*\ni f\mapsto (\Gamma f, \Gamma' f)\in \cG\times\cG$
is surjective. Given a boundary triple, one can describe in a constructive way all
self-adjoint extensions
of~$S$. A special role is played by the self-adjoint extension $H^0$ which
is the restriction of $S^*$ onto $\ker\Gamma$.
Furthermore, one can define the so-called Gamma field $\gamma(z)$
and the Weyl function $M(z)$. The Gamma field $z\mapsto \gamma(z)\in\ELL(\cG,\cH)$ is defined as follows:
for $z\in\CC\setminus\RR$ and $\xi\in\cG$, the vector $u:=\gamma(z)\xi$
is the unique solution of $(S^*-z)u=0$ with $\Gamma u=\xi$.
The Weyl function $M:\CC\setminus\RR\to\ELL(\cG)$ is defined by $M(z):=\Gamma'\gamma(z)$. It is
an operator-valued Herglotz function in the sense that
$z\mapsto \langle M(z)\xi,\xi\rangle$ is a scalar Herglotz function for any $\xi\in\cG$.
Moreover, $M(z)$ extends holomorphically to the whole resolvent set of $H^0$.

Now assume that $S$ is such that one can construct a boundary triple $(\CC^2,\Gamma,\Gamma')$,
$\Gamma f=(\Gamma_1 f,\Gamma_2 f)$, $\Gamma' f=(\Gamma'_1 f,\Gamma'_2 f)$. Then
the Weyl function is just a $2\times 2$ matrix function, $M(z)=\big(m_{jk}(z)\big)$.
Let us construct two new operators. The first one, denoted by $H^0$,
is the restriction of $S^*$ to $\ker \Gamma$, and it is a self-adjoint operator in $\cH$.
The second one, denoted by $L$, acts in $\cK:=\ell^2(\ZZ)\times \cH$ as
$L(f_n)=(S^* f_n)$ on the domain consisting of the vectors $(f_n)\in\cK$, $f_n\in \dom S^*$,
with $(S^*f_n)\in\cK$, satisfying the abstract gluing conditions
\begin{equation}
         \label{eq-bc}
\Gamma_1 f_{n+1}=\Gamma_2 f_n \text{ and }
\Gamma'_1 f_{n+1}+\Gamma'_2 f_n=0 \text{ for all } n\in\ZZ.
\end{equation}
It can be shown that $L$ is a self-adjoint operator in $\cK$, see e.g. \cite[Section~3]{KP12}.
The following result is easily follows from Theorem~3 in~\cite{KP12}:
\begin{prop}\label{prop4}
Assume that the resolvent set of $H^0$ contains a non-empty open interval
$I:=(\lambda_1,\lambda_2)\subset \RR$ in which the coefficient $m_{12}$ is real-valued and does not vanish, then
$\spec L\cap I=\big\{\lambda\in I: \, m(\lambda)\in\spec P\big\}$, where
$m(\lambda):=\dfrac{m_{11}(\lambda)+m_{22}(\lambda)}{m_{12}(\lambda)}$ and
and $P$ is the operator in $\ell^2(\ZZ)$ acting as $P f(n)=f(n-1)+f(n+1)$.
\end{prop}
One easily checks that $\spec P=[-2,2]$, which shows that the spectrum of $L$ in the gaps of $H^0$
is characterized by the condition $\big|m(\lambda)\big|\le 2$, and the computation of
$\langle M(z)\xi,\xi\rangle$ for $\xi=(1,\pm 1)$ shows that the functions
$h_\pm(z):=\big(m(z)\pm 2\big)/n(z)$, with $n(z):=1/m_{12}(z)$, are Herglotz.
This allows one to carry out the analysis of the points $\lambda$ at which $m(\lambda)=\pm 2$
using Lemma~\ref{prop1}.

Let us illustrate the preceding abstract machinery with the help of Hill's equation.
Take $\cH:=L^2(0,1)$ and denote by $S$ the operator defined on
$C_c^\infty(0,1)$ and acting as $u\mapsto -u''+Qu$. It is straightforward to check
that the adjoint $S^*$ is given by the same differential expression but acts
on $H^2(0,1)$. As a boundary triple for $S$ one can take $(\CC^2,\Gamma,\Gamma')$
with $\Gamma u=\big(u(0),u(1)\big)$ and $\Gamma' u=\big(u'(0),-u'(1)\big)$,
and the associated Weyl function is exactly $M(z)$ from \eqref{eq-mum}, see \cite[Example~15.3]{ksm}.
The space $\cK:=\ell^2(\ZZ)\otimes \cH$ can be simply viewed as $L^2(\RR)$
if one identifies $(u_n)\in \cK$ with the function $u$ on $\RR$ whose restriction
onto $(n,n+1)$ coincides with $u_n$, then the boundary conditions \eqref{eq-bc}
express simply the continuity of $u$ and $u'$ at each $n\in\ZZ$,
and the associated operator $L$ is the periodic Sch\"odinger operator
$-d^2/dx^2+Q$ on $\RR$. One simply checks that $m(z)=-\Delta(z)$ and $n(z):=-s(1;z)$,
and Proposition~\ref{prop4} shows that outside the set of zeros of $n$
the spectrum of $L$ is characterized by the condition $\Delta(\lambda)\in[-2,2]$.

Indeed, Hill's equation does not exhaust all situations covered by the above abstract scheme.
Admissible examples include e.g. operators with distributional coefficients and quasi-differential
operators~\cite{DM,HM,K,MM}, chains of coupled manifolds or graphs~\cite{BEG,EKW,MSB,KP12} or special matrix operators~\cite{bhsw,hsv}.
In general, the associated discriminant is not a holomorphic function, it can have
poles and more general singularities at the real axis, which depends
on the spectral properties of the respective operator~$H^0$, cf.~\cite{BMN02,BGP08,DM91}.
Our analysis is still applicable at least in the gaps of~$H^0$, as Lemma~\ref{prop1} is of a local nature and does not
involve any assumption on the global behavior of the participating functions.

\subsection*{Acknowledgments}
The work was partially supported by INSMI CNRS and ANR NOSEVOL.
We are grateful to Mark Malamud for bringing the papers~\cite{mm92,mn}
to our attention.

\end{document}